\documentclass[a4paper,leqno]{amsart}

\usepackage{latexsym}
\usepackage[english]{babel}
\usepackage{fancyhdr}
\usepackage[mathscr]{eucal}
\usepackage{amsmath}
\usepackage{mathrsfs}
\usepackage{amsthm}
\usepackage{amsfonts}
\usepackage{amssymb}
\usepackage{amscd}
\usepackage{bbm}
\usepackage{graphicx}
\usepackage{subcaption}
\usepackage{graphics}
\usepackage{latexsym}
\usepackage{color}

\usepackage{pifont}
\usepackage{booktabs} 

\newcommand{\ud}{\mathrm{d}}

\newcommand{\ii}{\mathrm{i}}
\newcommand{\cH}{\mathcal{H}}

\theoremstyle{plain}
\newtheorem{theorem}{Theorem}[section]

\newtheorem{corollary}[theorem]{Corollary}

\theoremstyle{definition}

\newtheorem{remark}[theorem]{Remark}

\numberwithin{equation}{section}

\begin{document}

\title[On absolute vs relative self-adjoint extension schemes]
{On a comparison between absolute and relative self-adjoint extension schemes}
\author[N.~A.~Caruso]{No\`e Angelo Caruso}
\address[N.~A.~Caruso]{Mathematical Institute, Silesian University in Opava (Czech Republic)}
\email{noe.caruso@math.slu.cz}
\author[A.~Michelangeli]{Alessandro Michelangeli}
\address[A.~Michelangeli]{Department of Mathematics and Natural Sciences, Prince Mohammad Bin Fahd University \\ Al Khobar 31952 (Saudi Arabia) \\
and Hausdorff Center for Mathematics, University of Bonn \\ Endenicher Allee 60 \\ D-53115 Bonn (Germany) \\ and TQT Trieste Institute for Theoretical Quantum Technologies, Trieste (Italy)}
\email{amichelangeli@pmu.edu.sa}
\author[A.~Ottolini]{Andrea Ottolini}
\address[A.~Ottolini]{Department of Mathematics, University of Washington \\ C-138 Padelford \\
Seattle, WA 98195-4350 (USA)}
\email{ottolini@uw.edu}


\begin{abstract}
 The problem of connecting the operator parameters that label the same self-adjoint extension of a given symmetric operator, respectively, within the `absolute' von Neumann extension scheme and the `relative' boundary-triplet-induced extension scheme (i.e., a la Kre{\u\i}n-Vi\v{s}ik-Birman) is discussed, and quantitative connections between the two parameters are established in the limit of deficiency spaces at complex spectral points converging to the deficiency space at a real spectral point.
 \end{abstract}

\date{\today}




\thanks{Partially supported by the Italian National Institute for Higher Mathematics INdAM and the Alexander von Humboldt Foundation, Bonn.. A.~M.~is most grateful to the Mathematical Institute at the Silesian University in Opava for the kind hospitality during the period in which this project was completed.
}

\maketitle


\section{Introduction}\label{sec:intro-main}

Consider a complex Hilbert space $\cH$, with scalar product $\langle\cdot,\cdot\rangle_{\cH}$ by convention linear in the second entry and anti-linear in the first, and with norm $\|\cdot\|_{\cH}$. Consider also a densely defined symmetric operator $S$ on $\cH$, possibly non-closed or unbounded. $S$ is then surely closable, and make on it the additional assumption that the resolvent set $\rho(\overline{S})$ of the operator closure $\overline{S}$ of $S$ contains a real point -- for concreteness the point zero, up to a shift of $S$ by a real multiple of the identity. Under these conditions a classical result by Calkin \cite[Theorem 2]{Calkin-1940} ensures that $S$ admits self-adjoint extensions, among which a distinguished extension $S_{\mathrm{D}}$ that has everywhere defined and bounded inverse on $\cH$. Moreover, another classical result by Krasnosel'ski\u{\i} and Kre{\u\i}n \cite{Krein-Krasnoselskii-Milman-1948} guarantees that the cardinal number 
\begin{equation}
 d(S)\;:=\;\mathrm{dim}\ker(S^*-z\mathbbm{1})\,,
\end{equation}
the deficiency index of $S$, is constant for each $z\in\mathbb{C}^+\cup\{0\}\cup\mathbb{C}^-$, i.e.,
\begin{equation}\label{eq:samedim}
 \mathrm{dim}\ker S^*\;=\;\mathrm{dim}\ker(S^*-z\mathbbm{1})\;=\;\mathrm{dim}\ker(S^*-\overline{z}\mathbbm{1})\qquad\forall z\in\mathbb{C}^+\,.
\end{equation}
Each of the closed subspaces of $\cH$ appearing in \eqref{eq:samedim} is a deficiency subspace of $S$ at the considered spectral point, and in fact according to von Neumann's extension theory \cite{vonNeumann1930} the identity $\mathrm{dim}\ker(S^*-z\mathbbm{1})=\mathrm{dim}\ker(S^*-\overline{z}\mathbbm{1})$ for one, and hence for all, $z\in\mathbb{C}^+$ is necessary and sufficient for $S$ to admit self-adjoint extensions on the same Hilbert space (see, e.g., \cite[Chapters 1 and 2]{GM-SelfAdj_book-2022} for details).

A paradigmatic example of the above set-up, that we shall keep in mind throughout, is the case when $S$ is densely defined and lower semi-bounded (hence, symmetric) on $\cH$, with strictly positive lower bound, i.e.,
\begin{equation}\label{eq:positivebottom}
 \mathfrak{m}(S)\;:=\;\inf_{\substack{f\in\mathcal{D}(S) \\ f\neq 0}}  \frac{\langle f,S f\rangle_{\cH}}{\|f\|^2_{\cH}}\;>\;0\,.
\end{equation}
 Such $S$ admits the Friedrichs extension $S_\mathrm{F}$, that distinguished self-adjoint extension of $S$ characterised by the property of being the largest, in the sense of operator ordering (see, e.g., \cite[Theorem 2.1]{GM-SelfAdj_book-2022}). One has $\mathfrak{m}(S_\mathrm{F})= \mathfrak{m}(S)>0$; thus, $S_\mathrm{F}$ has everywhere-defined and bounded inverse.

 In the above picture, it is known by von Neumann's theory that either the operator closure $\overline{S}$ is self-adjoint (i.e., $S$ is essentially self-adjoint), or $\overline{S}$, and hence $S$, admits an infinite family of distinct self-adjoint extensions, beside the distinguished one $S_\mathrm{D}$ (or $S_\mathrm{F}$ in the prototypical lower semi-bounded case), depending on the triviality or non-triviality of the dimension \eqref{eq:samedim}.

 In the non-trivial case of \emph{non-zero} deficiency index, there are two classical and alternative ways to parametrise the self-adjoint extension family for $S$, each one yielding a classification of all possible extensions.

 A first parametrisation is, in a sense, `absolute' and `non-canonical', and is provided by von Neumann's extension theory: for fixed $z\in\mathbb{C}^+$ there is a one-to-one correspondence $S_U\leftrightarrow U$ between self-adjoint extensions $S_U$ of $S$ in $\cH$ and unitary isomorphisms $U:\ker(S^*-z\mathbbm{1})\xrightarrow{\cong}\ker(S^*-\overline{z}\mathbbm{1})$, and the extension labelled by one such $U$ is given by
 \begin{equation}\label{eq:vNparam}
\begin{split}
\mathcal{D}(S_U)\;&:=\;\mathcal{D}(\overline{S})\,\dotplus\,(\mathbbm{1}-U)\ker(S^*-z\mathbbm{1})\,, \\
S_U(f+u-Uu)\;&:=\;S^*(f+u-Uu)\;=\;\overline{S}f+zu-\overline{z}\,Uu\,.
\end{split}
\end{equation}
 (see, e.g., \cite[Section 2.3]{GM-SelfAdj_book-2022} for details).

 A second alternative parametrisation is `relative' and `canonical', in that it canonically identifies each self-adjoint extension with respect to the distinguished extension $S_\mathrm{D}$. Although this is a parametrisation that stems from the old  Kre{\u\i}n-Vi\v{s}ik-Birman extension theory (we refer for details to the original works \cite{Krein-1947,Vishik-1952,Birman-1956}, as well as to \cite{KM-2015-Birman}, and to \cite[Section 2.7]{GM-SelfAdj_book-2022} and \cite[Theorem 2.17]{GM-SelfAdj_book-2022} in particular), and was later made explicit within Grubb's universal extension theory \cite{Grubb-1968} (see \cite[Chapter 13]{Grubb-DistributionsAndOperators-2009} and \cite[Section 4]{GMO-KVB2017} for details), it is today somehow fashionable to formulate it within the framework of modern boundary triplets theory \cite{Behrndt-Hassi-deSnoo-boundaryBook}. So, consider the classical decomposition formula 
 \begin{equation}\label{eq:II-VishBirDecomp}
   \mathcal{D}(S^*)\;=\;\mathcal{D}(\overline{S})\dotplus S_\mathrm{D}^{-1}\ker S^*\dotplus \ker S^* 
 \end{equation}
 (established by Vi\v{s}ik \cite[Eq.~(1.41)]{Vishik-1952} and Birman \cite[Theorem 1]{Birman-1956} in the special case \eqref{eq:positivebottom}, and straightforwardly re-provable for the present general setting -- see, e.g., \cite[Proposition 2.2 and 2.3]{GM-SelfAdj_book-2022}). One sees that \eqref{eq:II-VishBirDecomp} induces two maps $\Gamma_0,\Gamma_1:\mathcal{D}(S^*)\to\ker S^*$ defined on generic $g=f+S_\mathrm{D}^{-1} u_1+u_0\in\mathcal{D}(S^*)$ by
 \begin{equation}
  \begin{split}
  \Gamma_0 g\;&:=\;u_0\;=\; g-S_\mathrm{D}^{-1}S^*g\,, \\
   \Gamma_1 g\;&:=\;u_1\;=\;P_{\ker S^*}S^* g 
  \end{split}
 \end{equation}
 ($P_{\ker S^*}$ denoting the orthogonal projection from $\cH$ onto $\ker S^*$). The triplet 
 \begin{equation}\label{eq:thetriplet}
  (\ker S^*,\Gamma_0,\Gamma_1)
 \end{equation}
 is recognised to be an actual boundary triplet for $S^*$ (see, e.g., \cite[Example 14.6]{schmu_unbdd_sa}) for which the general boundary-triplet-based classification theorem of self-adjoint extensions (see, e.g., \cite[Theorem 14.10]{schmu_unbdd_sa}) produces the following parametrisation (see, e.g., \cite[Theorem 14.12]{schmu_unbdd_sa}): there is a one-to-one correspondence $S_T\leftrightarrow T$ between self-adjoint extensions $S_T$ of $S$ and self-adjoint operators $T$ acting in a closed subspace of $\ker S^*$, and the extension labelled by one such $T$ is given by
 \begin{equation}\label{eq:ST-2}
\begin{split}
\mathcal{D}(S_T)\;&:=\;\left\{f+S_\mathrm{D}^{-1}(Tu+w)+u\left|
\begin{array}{c}
f\in\mathcal{D}(\overline{S})\,,\;u\in\mathcal{D}(T) \\
w\in\ker S^*\cap\mathcal{D}(T)^\perp
\end{array}
\right.\right\}, \\
S_T\;&=\;S^*\upharpoonright\mathcal{D}(S_T)\,,
\end{split}
\end{equation}
i.e.,
\begin{equation}
 S_T(f+S_\mathrm{D}^{-1}(Tu+w)+u)\;=\;\overline{S}f+Tu+w\,.
\end{equation}

 The purpose of this note is to provide \emph{a connection between the parametrisations \eqref{eq:vNparam} and \eqref{eq:ST-2}} that be at the same time conceptual, in the sense of structural links between the two pictures, and quantitative, in fact through an explicit limit procedure that we are going to discuss.

 Underlying the above general goal there is a two-fold question. On the one hand, for a generic self-adjoint extension $\widetilde{S}$ of $S$ labelled respectively as $\widetilde{S}=S_U$ according to \eqref{eq:vNparam} (for given deficiency subspaces $\ker(S^*-z\mathbbm{1})$ and $\ker(S^*-\overline{z}\mathbbm{1})$) and $\widetilde{S}=S_T$ according to \eqref{eq:ST-2}, one inquires on the connection between $U$ and $T$, each being 	uniquely determined by $\widetilde{S}$ (observe that such two labelling operators are completely different in nature -- a unitary and a self-adjoint, and acting in different spaces). On the other hand, an analogous issue can be posed point-wise, thus inquiring how to connect the canonical components of an element $g\in\mathcal{D}(\widetilde{S})$ written as 
 \[
  f+S_{\mathrm{D}}^{-1}(Tu+w)+u\;=\;g\;=\;\widetilde{f}+\widetilde{u}-U\widetilde{u}
 \]
 with self-explanatory notation with respect to \eqref{eq:vNparam} and \eqref{eq:ST-2}.

 It is fairly natural to expect that the von Neumann extension picture \eqref{eq:vNparam} takes approximately the form of \eqref{eq:ST-2} when the non-real spectral point $z$ chosen to identify the deficiency subspaces is close to the spectral point zero, as the comparison between $\ker(S^*-z\mathbbm{1})$ and $\ker S^*$ suggests. We shall discuss this by choosing $z=\ii\varepsilon$, $\varepsilon>0$, in the limit $\varepsilon\downarrow 0$ and we shall see that it is already at the level of the adjoint $S^*$ that there is an explicit connection between the decomposition of $\mathcal{D}(S^*)$ at the spectral points $z=\ii\varepsilon$ and at $z=0$. This is the object of Section \ref{sec:impretriplet}. Then in Section \ref{sec:convparam} the actual convergence of the two self-adjoint extension parametrisations is shown, connecting $U$ and $T$ as $\varepsilon\downarrow 0$. Two elucidative examples of our results are discussed in Section \ref{eq:examples}.

 \section{Imaginary pre-triplet for $S^*$}\label{sec:impretriplet}
 
 Preliminary to the actual comparison between the self-adjoint extension parametrisations \eqref{eq:vNparam} and \eqref{eq:ST-2}, we discuss the connection between the decompositions of $\mathcal{D}(S^*)$ at real and at imaginary spectral points.

 The following shall be assumed:
 \begin{equation}\label{eq:assumptions}
  \begin{split}
   & \textrm{$S$ is a densely defined symmetric operator in $\cH$} \\
   & \textrm{such that }\quad 0\in\rho(\overline{S})\quad\textrm{ and }\quad d(S)\geqslant 1\,,
  \end{split}
 \end{equation}
 $d(S)$ being the deficiency index defined in \eqref{eq:positivebottom}. As already mentioned, a special (and frequent in applications) version of \eqref{eq:assumptions} is
 \begin{equation}\label{eq:assumptions-lowerbdd}
 \begin{split}
  & \textrm{$S$ is a densely defined and lower-semi-bounded operator in $\cH$} \\
  & \textrm{such that }\quad \mathfrak{m}(S)\;=\inf_{\substack{f\in\mathcal{D}(S) \\ u\neq 0}}  \frac{\langle f,S f\rangle_{\cH}}{\|f\|^2_{\cH}}\;>\;0\;\;\textrm{ and }\;\; d(S)\geqslant 1\,.
 \end{split}
 \end{equation}
 In either case $S$ admits an infinite family of distinct self-adjoint extensions, including a distinguished one, denoted by $S_\mathrm{D}$, having everywhere defined and bounded inverse in $\cH$. In the special case \eqref{eq:assumptions-lowerbdd} the role of distinguished extension is played by the Friedrichs extension $S_\mathrm{F}$ of $S$.

 For arbitrary $\varepsilon>0$, under \eqref{eq:assumptions} one has
 \begin{equation}\label{eq:vNformula}
  \mathcal{D}(S^*)\;=\;\mathcal{D}(\overline{S})\dotplus\ker(S^*-\ii\varepsilon\mathbbm{1})\dotplus\ker(S^*+\ii\varepsilon\mathbbm{1})
 \end{equation}
 and also
 \begin{equation}\label{eq:VBdec}
  \mathcal{D}(S^*)\;=\;\mathcal{D}(\overline{S})\dotplus S_\mathrm{D}^{-1}\ker S^*\dotplus \ker S^*\,.
 \end{equation}
 The decomposition \eqref{eq:vNformula} is the classical von Neumann formula at the spectral point $z=\ii\varepsilon$, and in fact it is valid for generic densely defined and symmetric $S$; \eqref{eq:VBdec} is in practice the classical Vi\v{s}ik-Birman decomposition (originally established in the special case \eqref{eq:assumptions-lowerbdd}).

 We have already introduced the two boundary maps
 \begin{equation}\label{eq:defmaps}
  \Gamma_0\;=\;\mathbbm{1}-S_\mathrm{D}^{-1}S^*\,,\qquad \Gamma_1\;=\;P_{\ker S^*}S^*\,,
 \end{equation}
 as $\mathcal{D}(S^*)\to\ker S^*$ maps, and next to them let us define
 \begin{equation}\label{eq:defmapseps}
  \begin{split}
   \Gamma_{0,\varepsilon}\;&:=\;\Gamma_0\,\Upsilon_\varepsilon\,, \\ 
   \Upsilon_\varepsilon &:=\;\frac{\Gamma_{1,\varepsilon}^- - \Gamma_{1,\varepsilon}^+}{2\,\ii\,\varepsilon}\,, \\
   \Gamma_{1,\varepsilon}^-\;&:=\;P_{\ker(S^*-\ii\varepsilon)}(S^*+\ii\varepsilon\mathbbm{1})\,, \\
   \Gamma_{1,\varepsilon}^+\;&:=\;P_{\ker(S^*+\ii\varepsilon)}(S^*-\ii\varepsilon\mathbbm{1})\,,
  \end{split}
 \end{equation}
 again defined on $\mathcal{D}(S^*)$, and with the customary notation for orthogonal projections $P_W$ onto closed subspaces $W$ of $\cH$.

 In view of the results that will follow, it makes sense to refer to
 \begin{equation}
  \big(\ker(S^*+\ii\varepsilon\mathbbm{1})\,,\,\Gamma_{0,\varepsilon} \,,\,\Gamma_{1,\varepsilon}^-\big)
 \end{equation}
 (a concise form for the actual quintuple consisting of $\ker(S^*-\ii\varepsilon\mathbbm{1})$, $\ker(S^*+\ii\varepsilon\mathbbm{1})$, $\Gamma_{0,\varepsilon}$, $\Gamma_{1,\varepsilon}^-$, and $\Gamma_{1,\varepsilon}^+$) as the \emph{imaginary pre-triplet} associated, at the chosen $\varepsilon$, to the boundary triplet $(\ker S^*,\Gamma_0,\Gamma_1)$ already encountered in \eqref{eq:thetriplet} above.

 One has indeed the following.

 \begin{theorem}\label{thm:mapsconvergence}
  Under the assumptions \eqref{eq:assumptions} one has:
  \begin{enumerate}
   \item[(i)] $\big\| P_{\ker(S^*\pm\ii\varepsilon\mathbbm{1})})-P_{\ker S^*}\big\|_{\mathrm{op}}\,\leqslant\,\varepsilon\,\|\overline{S}^{-1}\|_{\mathrm{op}}$\,;   \\
   in particular, $P_{\ker(S^*\pm\ii\varepsilon\mathbbm{1})}\xrightarrow{\;\varepsilon\downarrow 0\;} P_{\ker S^*}$ in operator norm.
   \item[(ii)] $\|\Gamma_{1,\varepsilon}^\pm\,g-\Gamma_1 g\|_{\cH}\,\leqslant\,\varepsilon\big(\|g\|_{\cH}+\|\overline{S}^{-1}\|_{\mathrm{op}}\,\|S^*g\|_{\cH}\big)$ for any $g\in\mathcal{D}(S^*)$; \\ in particular, $\Gamma_{1,\varepsilon}^\pm\xrightarrow{\;\varepsilon\downarrow 0\;}\Gamma_1$ strongly on $\mathcal{D}(S^*)$.
   \item[(iii)] $\|\Gamma_{0,\varepsilon}g - \Gamma_0 g\|_{\cH}\,\leqslant\,\varepsilon\,\big(\|\overline{S}^{-1}\|_{\mathrm{op}}+\|S_\mathrm{D}^{-1}\|_{\mathrm{op}}\big)\big(\|g\|_{\cH}+\|\overline{S}^{-1}\|_{\mathrm{op}}\,\|S^*g\|_{\cH}\big)$ for any $g\in\mathcal{D}(S^*)$; in particular,   
   $\Gamma_{0,\varepsilon}\xrightarrow{\;\varepsilon\downarrow 0\;}\Gamma_0$ strongly on $\mathcal{D}(S^*)$.
  \end{enumerate}
 \end{theorem}

 Concerning the proof of part (i), in view of the Hilbert space orthogonal decomposition
 \begin{equation}
  \cH\;=\;\mathrm{ran}\overline{S}\oplus\ker S^*\;=\;\mathrm{ran}(\overline{S}\mp\ii\varepsilon\mathbbm{1})\oplus\ker(S^*\pm\ii\varepsilon\mathbbm{1})
 \end{equation}
 (owing to assumptions \eqref{eq:assumptions} the spaces $\mathrm{ran}\overline{S}$ and $\mathrm{ran}(\overline{S}\mp\ii\varepsilon\mathbbm{1})$ are closed), the claim is equivalent to $\big\|P_{\mathrm{ran}(\overline{S}\mp\ii\varepsilon\mathbbm{1})}-P_{\mathrm{ran}\overline{S}}\big\|_{\mathrm{op}}\,\leqslant\,\varepsilon\,\|\overline{S}^{-1}\|_{\mathrm{op}}$. In turn, it is known \cite[Chapt.~4, \S 2]{Kato-perturbation} that the latter claim is equivalent to the
 fact that the gap metric distance between $\mathrm{ran}(\overline{S}\mp\ii\varepsilon\mathbbm{1})$ and $\mathrm{ran}\overline{S}$ does not exceed $\varepsilon\,\|\overline{S}^{-1}\|_{\mathrm{op}}$, in the sense of the  gap metric $\widehat{\delta}$ between closed subspaces of $\cH$. Recall that for any two closed subspaces $U,V\subset\cH$ the gap $\widehat{\delta}(U,V)$ is defined by
 \begin{equation}\label{eq:dhat}
  \begin{split}
   \widehat{\delta}(U,V)\;&:=\;\max\{\delta(U,V),\delta(V,U)\}\,, \\
   \delta(U,V)\;&:=\;\sup_{\substack{ u\in U \\ \|u\|_{\cH}=1 }}\inf_{\substack{ v\in V}}\|u-v\|_{\cH}\;=\;\sup_{\substack{ u\in U \\ \|u\|_{\cH}=1 }}\mathrm{dist}(u,V)\,,
  \end{split}
 \end{equation}
 and with the tacit assignment $\delta(\{0\},V):=0$.
 This is the classical notion of `opening' between (closed) subspaces introduced by Kre\u{\i}n and Krasnosel$'$ski\u{\i} 
 \cite{Krein-Krasnoselskii-1947} (with generalisation to the Banach space setting -- see, e.g., \cite{Krein-Krasnoselskii-Milman-1948}, and \cite[\S 34]{Akhiezer-Glazman-1961-1993}): as a matter of fact \cite{Gohberg-Markus-1959}, the set of all closed subspaces of $\cH$ is complete in the metric $ \widehat{\delta}$.

 \begin{proof}[Proof of Theorem \ref{thm:mapsconvergence}, part (i)]
 The goal is to prove that 
 \[\tag{*}\label{deltahatlimit}
  \widehat{\delta}(\mathrm{ran}(\overline{S}\mp\ii\varepsilon\mathbbm{1}),\mathrm{ran}\overline{S}))\;\leqslant\;\varepsilon\,\|\overline{S}^{-1}\|_{\mathrm{op}}\,.
 \]
 Let us treat the `$-$' case, the `$+$' case being completely analogous. So one needs to control both
 \[
  \delta(\mathrm{ran}(\overline{S}-\ii\varepsilon\mathbbm{1}),\mathrm{ran}\overline{S})
  \qquad\textrm{and}\qquad \delta(\mathrm{ran}\overline{S},\mathrm{ran}(\overline{S}-\ii\varepsilon\mathbbm{1}))\,.
 \]
  Concerning the first, consider an arbitrary $u=\overline{S}f-\ii\varepsilon f\in \mathrm{ran}(\overline{S}-\ii\varepsilon\mathbbm{1})$ for some $f\in\mathcal{D}(\overline{S})$, and such that $\|u\|_{\cH}=1$.
  By assumptions \eqref{eq:assumptions} one has $f=\overline{S}^{-1}\overline{S}f$, whence 
  \[
   \|\overline{S}f\|_{\cH}\;\geqslant\;\|\overline{S}^{-1}\|_{\mathrm{op}}^{-1}\,\|f\|_{\cH}\,.
  \]
  Using this, and the symmetry of $\overline{S}$, 
  \[
   1\;=\;\|u\|_{\cH}^2\;=\;\|(\overline{S}-\ii\varepsilon\mathbbm{1})f\|_{\cH}^2\;=\;\|\overline{S}f\|_{\cH}^2+\varepsilon^2\|f\|_{\cH}^2\;>\;\|\overline{S}^{-1}\|_{\mathrm{op}}^{-2}\,\|f\|_{\cH}^2\,,
  \]
  i.e.,
  \[
   \|f\|_{\cH}\;<\;\|\overline{S}^{-1}\|_{\mathrm{op}}\,.
  \]
  Therefore,
  \[
   \mathrm{dist}(u,\mathrm{ran}\overline{S})\;\leqslant\;\|u-\overline{S}f\|_{\cH}\;=\;\varepsilon\|f\|_{\cH}\;<\;\varepsilon\,\|\overline{S}^{-1}\|_{\mathrm{op}}\,,
  \]
  and 
   \[
    \delta(\mathrm{ran}(\overline{S}-\ii\varepsilon\mathbbm{1}),\mathrm{ran}\overline{S})\;=\;\sup_{\substack{ u\in \mathrm{ran}(\overline{S}-\ii\varepsilon\mathbbm{1}) \\ \|u\|_{\cH}=1 }}\mathrm{dist}(u,\mathrm{ran}\overline{S})\;\leqslant\;\varepsilon\,\|\overline{S}^{-1}\|_{\mathrm{op}}\,.
   \]
  Concerning the control of $\delta(\mathrm{ran}\overline{S},\mathrm{ran}(\overline{S}-\ii\varepsilon\mathbbm{1}))$, consider now an arbitrary $u=\overline{S}f\in\mathrm{ran}\overline{S}$ for some $f\in\mathcal{D}(\overline{S})$ and such that $\|u\|_{\cH}=1$. Then $\|f\|_{\cH}\leqslant\|\overline{S}^{-1}\|_{\mathrm{op}}$ and
  \[
   \mathrm{dist}(u,\mathrm{ran}(\overline{S}-\ii\varepsilon\mathbbm{1}))\;\leqslant\;\|u-(\overline{S}f-\ii\varepsilon f)\|_{\cH}\;=\;\varepsilon\|f\|_{\cH}\;\leqslant\;\varepsilon\,\|\overline{S}^{-1}\|_{\mathrm{op}}\,,
  \]
  whence
  \[
   \delta(\mathrm{ran}\overline{S},\mathrm{ran}(\overline{S}-\ii\varepsilon\mathbbm{1}))\;=\;\sup_{\substack{ u\in \mathrm{ran}\overline{S} \\ \|u\|_{\cH}=1 }}\mathrm{dist}(u,\mathrm{ran}(\overline{S}-\ii\varepsilon\mathbbm{1}))\;\leqslant\;\varepsilon\,\|\overline{S}^{-1}\|_{\mathrm{op}}\,.
  \]
  This shows finally
  \[
    \widehat{\delta}(\mathrm{ran}(\overline{S}-\ii\varepsilon\mathbbm{1}),\mathrm{ran}\overline{S}))\;\leqslant\;\varepsilon\,\|\overline{S}^{-1}\|_{\mathrm{op}}\,,
  \]
 so the bound \eqref{deltahatlimit} is proved (in the `$-$' case). 
 \end{proof}

 Passing now to the proof of Theorem \ref{thm:mapsconvergence}(ii) and (iii), decompose for every $\varepsilon>0$ a generic $g\in\mathcal{D}(S^*)$ as 
 \begin{equation}\label{eq:gdecSstar}
  f+S_{\mathrm{D}}^{-1}u_1+u_0\;=\;g\;=\;f_\varepsilon+u_\varepsilon-v_\varepsilon
 \end{equation}
 according to \eqref{eq:vNformula} and \eqref{eq:VBdec}, for unique vectors $f,f_\varepsilon\in\mathcal{D}(\overline{S})$, $u_0,u_1\in\ker S^*$, $u_\varepsilon\in\ker(S^*-\ii\varepsilon\mathbbm{1})$, and $v_\varepsilon\in\ker(S^*+\ii\varepsilon\mathbbm{1})$, and observe that 
 \begin{equation}\label{eq:g0g1action}
  \Gamma_0 g\;=\;u_0\,, \qquad \Gamma_1 g\;=\; u_1
 \end{equation}
 and 
 \begin{equation}\label{eq:gammaepsespl}
  \Gamma_{1,\varepsilon}^-g\;=\;2\,\ii\,\varepsilon\, u_\varepsilon\,,\qquad \Gamma_{1,\varepsilon}^+g\;=\;2\,\ii\,\varepsilon\,v_\varepsilon\,.
 \end{equation}
 As a consequence,
 \begin{equation}\label{eq:uepsminusveps}
  \Upsilon_\varepsilon g \;=\; u_\varepsilon-v_\varepsilon
 \end{equation}
  and 
 \begin{equation}\label{eq:Supsilon}
  S^*\Upsilon_\varepsilon\;=\;\frac{1}{2}(\Gamma_{1,\varepsilon}^-+\Gamma_{1,\varepsilon}^+) \qquad\textrm{on $\mathcal{D}(S^*)$}\,.
 \end{equation}
 Indeed, from \eqref{eq:gdecSstar}-\eqref{eq:gammaepsespl} one has 
  \[
   S^*\Upsilon_\varepsilon g\;=\;S^*(u_\varepsilon-v_\varepsilon)\;=\;\ii\,\varepsilon\,u_\varepsilon+\ii\,\varepsilon\,v_\varepsilon\;=\;\frac{1}{2}\,\Gamma_{1,\varepsilon}^- g+\frac{1}{2}\,\Gamma_{1,\varepsilon}^+ g\,.
  \]
  
  \begin{proof}[Proof of Theorem \ref{thm:mapsconvergence}, parts (ii) and (iii)]
   By part (i) 
   and by the obvious identity $\|(S^*\pm\ii\varepsilon\mathbbm{1})g - S^*g\|_{\cH}=\varepsilon\|g\|_{\cH}$, 
   \[
    \big\|P_{\ker(S^*\mp\ii\varepsilon)}(S^*\pm\ii\varepsilon\mathbbm{1})g-P_{\ker S^*} S^*g\big\|_{\cH}\;\leqslant\;\varepsilon\big(\|g\|_{\cH}+\|\overline{S}^{-1}\|_{\mathrm{op}}\,\|S^*g\|_{\cH}\big)\,.
   \]
   In view of \eqref{eq:defmaps}-\eqref{eq:defmapseps} this establishes part (ii). 
   Next, one has
   \[\tag{a}\label{eq:afterSstar}
    \overline{S}f+u_1\;=\;\overline{S}f_\varepsilon+\ii\,\varepsilon\, u_\varepsilon+\ii\,\varepsilon\, v_\varepsilon\;=\;\overline{S}f_\varepsilon+S^*\Upsilon_\varepsilon g\,,
   \]
  the first identity above following from \eqref{eq:gdecSstar}, and the second from \eqref{eq:gammaepsespl}-\eqref{eq:Supsilon}. Since part (ii) and \eqref{eq:Supsilon} imply
  \[\tag{b}\label{eq:Sstarupsilon}
   \|S^*\Upsilon_\varepsilon g - u_1 \|_{\cH}\;=\;\|S^*\Upsilon_\varepsilon g - \Gamma_1 g \|_{\cH}\;\leqslant\;\varepsilon\,\big(\|g\|_{\cH}+\|\overline{S}^{-1}\|_{\mathrm{op}}\,\|S^*g\|_{\cH}\big)\,,
  \]
  one then deduces
  \[\tag{c}\label{eq:SDm1Sstar}
   \|S_\mathrm{D}^{-1}S^*\Upsilon_\varepsilon g- S_\mathrm{D}^{-1} u_1\|_{\cH}\;\leqslant\;\varepsilon\,\|S_\mathrm{D}^{-1}\|_{\mathrm{op}}\big(\|g\|_{\cH}+\|\overline{S}^{-1}\|_{\mathrm{op}}\,\|S^*g\|_{\cH}\big)
  \]
  from \eqref{eq:Sstarupsilon},
  \[\tag{d}\label{eq:buh}
   \|\overline{S}f_\varepsilon-\overline{S}f\|_{\cH}\;\leqslant\;\varepsilon\,\big(\|g\|_{\cH}+\|\overline{S}^{-1}\|_{\mathrm{op}}\,\|S^*g\|_{\cH}\big)
  \]
  from \eqref{eq:afterSstar} and \eqref{eq:Sstarupsilon}, and 
   \[\tag{e}\label{eq:fepsconftof}
     \|f_\varepsilon - f\|_{\cH}\;\leqslant\;\varepsilon\,\|\overline{S}^{-1}\|_{\mathrm{op}}\big(\|g\|_{\cH}+\|\overline{S}^{-1}\|_{\mathrm{op}}\,\|S^*g\|_{\cH}\big)
  \]
  from \eqref{eq:buh} and the bounded invertibility of $\overline{S}$. 
  In turn, \eqref{eq:fepsconftof} above, \eqref{eq:gdecSstar}, and \eqref{eq:uepsminusveps} imply 
   \[\tag{f}\label{eq:upsepsgto}
    \begin{split}
     \|\Upsilon_\varepsilon g-(S_{\mathrm{D}}^{-1}u_1+u_0)\|_{\cH}\;& =\;\|f_\varepsilon - f\|_{\cH} \\
     &\leqslant\;\varepsilon\,\|\overline{S}^{-1}\|_{\mathrm{op}}\big(\|g\|_{\cH}+\|\overline{S}^{-1}\|_{\mathrm{op}}\,\|S^*g\|_{\cH}\big)\,.
    \end{split}
   \]
   By re-writing  $\Gamma_{0,\varepsilon}g=\Upsilon_\varepsilon g-S_\mathrm{D}^{-1}S^*\Upsilon_\varepsilon g$ from \eqref{eq:defmaps}-\eqref{eq:defmapseps}, and $\Gamma_0 g=u_0=(S_{\mathrm{D}}^{-1}u_1+u_0)-S_{\mathrm{D}}^{-1}u_1$, \eqref{eq:SDm1Sstar} and \eqref{eq:upsepsgto} are combined so as to yield
  \[
   \begin{split}
    \|\Gamma_{0,\varepsilon}g-\Gamma_0 g\|_{\cH}\;&=\;\big\|\big(\Upsilon_\varepsilon g-(S_{\mathrm{D}}^{-1}u_1+u_0)\big)-\big(S_\mathrm{D}^{-1}S^*\Upsilon_\varepsilon g- S_\mathrm{D}^{-1} u_1\big)\big\|_{\cH} \\
    &\leqslant\;\varepsilon\,\big(\|\overline{S}^{-1}\|_{\mathrm{op}}+\|S_\mathrm{D}^{-1}\|_{\mathrm{op}}\big)\big(\|g\|_{\cH}+\|\overline{S}^{-1}\|_{\mathrm{op}}\,\|S^*g\|_{\cH}\big)
   \end{split}
%
  \] 
  and also part (iii) is proved.
  \end{proof}

  The statement and the proof of Theorem \ref{thm:mapsconvergence} imply the following.	
  
  \begin{corollary}\label{cor:limitsSstar}
   Under the assumptions \eqref{eq:assumptions} write for any $\varepsilon>0$ an arbitrary $g\in\mathcal{D}(S^*)$ as in the decomposition \eqref{eq:gdecSstar}. Then, as $\varepsilon\downarrow 0$,
   \begin{eqnarray}
    (\mathbbm{1}-S_{\mathrm{D}}^{-1}S^*)(u_\varepsilon-v_\varepsilon)\!&=&\! u_0+O(\varepsilon)\,, \\
 \Gamma_{1,\varepsilon}^- g\;=\;2\,\ii\,\varepsilon\, u_\varepsilon \!&=&\! u_1+O(\varepsilon)\,, \\
 \Gamma_{1,\varepsilon}^+ g\;=\;2\,\ii\,\varepsilon\, v_\varepsilon \!&=&\! u_1+O(\varepsilon)\,, \\
     \Upsilon_\varepsilon g\;=\;\frac{\Gamma_{1,\varepsilon}^- - \Gamma_{1,\varepsilon}^+}{2\,\ii\,\varepsilon}\,g	 \;=\;u_\varepsilon-v_\varepsilon\!&=&\!S_{\mathrm{D}}^{-1}u_1+u_0+O(\varepsilon)\,, \\
         S^*\Upsilon_\varepsilon g\;=\;S^*(u_\varepsilon-v_\varepsilon)\!&=&\! u_1+O(\varepsilon)\,, \\
    f_\varepsilon \!&=&\! f+O(\varepsilon)\,, \\
    \overline{S}f_\varepsilon \!&= &\! \overline{S}f+O(\varepsilon)
   \end{eqnarray}
  in the norm topology of $\cH$.  
  \end{corollary}

  It is worth remarking the following behaviours. For each fixed $g\in\mathcal{D}(S^*)$ the `regular' component $f_\varepsilon$ relative to (the von Neumann decomposition at) the spectral point $\ii\varepsilon$ (see \eqref{eq:gdecSstar} above) approximates the `regular' component $f$ in the graph norm of $\overline{S}$, i.e.,
  \begin{equation}
   \|f-f_\varepsilon\|_{\Gamma(S^*)}\;=\;\big(\|f-f_\varepsilon\|_{\cH}^2+\|\overline{S}(f-f_\varepsilon)\|_{\cH}^2\big)^{\frac{1}{2}}\;\stackrel{\varepsilon\downarrow 0}{=}\;O(\varepsilon)\,.
  \end{equation}
The `singular' components $u_\varepsilon$ and $v_\varepsilon$ of $g$ relative to $\ker(S^*\mp\ii\varepsilon\mathbbm{1})$ separately have divergent norms $\| u_\varepsilon\|_{\cH}\sim\varepsilon^{-1}$ and $\|v_\varepsilon\|_{\cH}\sim\varepsilon^{-1}$, a divergence that cancels out in their difference $u_\varepsilon-v_\varepsilon$: it too has $S^*$-graph-norm limit, explicitly, 
  \begin{equation}
   \big\|(u_\varepsilon-v_\varepsilon)-(S_{\mathrm{D}}^{-1}u_1+u_0)\big\|_{\Gamma(S^*)}\;\stackrel{\varepsilon\downarrow 0}{=}\;O(\varepsilon)\,.
  \end{equation}

 \section{Convergence of extension parametrisations.}\label{sec:convparam}

 The analysis of the previous Section can be summarised by saying that under the assumptions \eqref{eq:assumptions} and in the precise sense of Theorem \ref{thm:mapsconvergence} the decompositions of $\mathcal{D}(S^*)$ respectively in terms of the deficiency spaces $\ker(S^*\mp\ii\varepsilon\mathbbm{1})$ and $\ker S^*$ gives rise, as $\varepsilon\downarrow 0$, to the convergence 
 \begin{equation}
  \big(\ker(S^*+\ii\varepsilon\mathbbm{1})\,,\,\Gamma_{0,\varepsilon} \,,\,\Gamma_{1,\varepsilon}^-\big)\;\xrightarrow{\,\varepsilon\downarrow0\,}\; (\ker S^*,\Gamma_0,\Gamma_1)
 \end{equation}
 of the imaginary pre-triplet associated to the canonical triplet \eqref{eq:thetriplet}. This \emph{structural} property implies an analogous convergence, in a sense now to be made precise, of the (von Neumann) self-adjoint extension parametrisation \eqref{eq:vNparam} to the (Kre{\u\i}n-Vi\v{s}ik-Birman-Grubb) self-adjoint extension parametrisation \eqref{eq:ST-2}.

 As anticipated in the introduction, the question of conceptual relevance we are interested in concerns the link, for an arbitrary fixed self-adjoint extension $\widetilde{S}$ of $S$, between the operator label $U_\varepsilon$, $\varepsilon>0$, the unitary $U_\varepsilon:\ker(S^*-\ii\varepsilon\mathbbm{1})\xrightarrow{\cong}\ker(S^*+\ii\varepsilon\mathbbm{1})$ uniquely associated to $\widetilde{S}\equiv S_{U_\varepsilon}$ via
  \begin{equation}\label{eq:vNparameps}
\begin{split}
\mathcal{D}(S_{U_\varepsilon})\;&:=\;\mathcal{D}(\overline{S})\,\dotplus\,(\mathbbm{1}-U_\varepsilon)\ker(S^*-\ii\varepsilon\mathbbm{1})\,, \\
S_{U_\varepsilon}(f_\varepsilon+u_\varepsilon-U_\varepsilon u_\varepsilon)\;&:=\;\overline{S} f_\varepsilon +\ii\varepsilon u_\varepsilon+\ii\varepsilon\,U_\varepsilon u_\varepsilon\,,
\end{split}
\end{equation}
 and the operator label $T$, the self-adjoint operator associated to $\widetilde{S}\equiv S_T$ via 
  \begin{equation}\label{eq:ST-2again}
\begin{split}
&\mathcal{D}(S_T)\;:=\;\left\{f+S_\mathrm{D}^{-1}(Tu+w)+u\left|
\begin{array}{c}
f\in\mathcal{D}(\overline{S})\,,\;u\in\mathcal{D}(T) \\
w\in\ker S^*\cap\mathcal{D}(T)^\perp
\end{array}
\right.\right\}, \\
&  S_T(f+S_\mathrm{D}^{-1}(Tu+w)+u)\;:=\;\overline{S}f+Tu+w\,.
\end{split}
\end{equation}

For concreteness one may ask: from the sole knowledge of $U_\varepsilon\equiv U_\varepsilon^{(\widetilde{S})}$, namely its action on $\ker(S^*-\ii\varepsilon\mathbbm{1})$, how to reconstruct the knowledge (domain and action) of $T\equiv T^{(\widetilde{S})}$ asymptotically as $\varepsilon\downarrow 0$ ?

And conversely: given the information on $T\equiv T^{(\widetilde{S})}$, how to reconstruct (an approximate information of) $U_\varepsilon\equiv U_\varepsilon^{(\widetilde{S})}$ for sufficiently small $\varepsilon>0$ ?

 The answers are given here, based on the findings of the previous Section. To begin with, a further corollary to Theorem \ref{thm:mapsconvergence}, immediately derivable from \eqref{eq:g0g1action}-\eqref{eq:gammaepsespl} and from Corollary \ref{cor:limitsSstar} and which we now give separate status of a theorem, is the following.

\begin{theorem}
 Assume \eqref{eq:assumptions} and let $\widetilde{S}$ be a self-adjoint extension of $S$. For arbitrary $\varepsilon>0$ let $U_\varepsilon\equiv U_\varepsilon^{(\widetilde{S})}$ and $T\equiv T^{(\widetilde{S})}$ be the extension parameters of $\widetilde{S}$ given, respectively, by \eqref{eq:vNparameps} and \eqref{eq:ST-2again}. Decompose an arbitrary $g\in\mathcal{D}(\widetilde{S})$ accordingly, i.e.,
 \begin{equation}\label{eq:gdecomposedaccordingly}
  f^{(g)}+S_{\mathrm{D}}^{-1}(Tu^{(g)}+w^{(g)})+u^{(g)}\;=\;g\;=\;f_\varepsilon^{(g)}+u_\varepsilon^{(g)}- U_\varepsilon u_\varepsilon^{(g)}
 \end{equation}
 for $f^{(g)},f_\varepsilon^{(g)}\in\mathcal{D}(\overline{S})$, $u^{(g)}\in\mathcal{D}(T)$, $w^{(g)}\in\ker S^*\cap\mathcal{D}(T)^\perp$, and $u_\varepsilon^{(g)}\in\ker(S^*-\ii\varepsilon\mathbbm{1})$, all uniquely determined by $g$. Then:
\begin{eqnarray}
 u_\varepsilon^{(g)} \!&=&\! (2\,\ii\,\varepsilon)^{-1} P_{\ker(S^*-\ii\varepsilon\mathbbm{1})}(\widetilde{S}+\ii\varepsilon\mathbbm{1})g\,, \label{eq:uvareps} \\
 U_\varepsilon u_\varepsilon^{(g)} \!&=&\! (2\,\ii\,\varepsilon)^{-1} P_{\ker(S^*+\ii\varepsilon\mathbbm{1})}(\widetilde{S}-\ii\varepsilon\mathbbm{1})g\,, \label{eq:Uuvareps} \\
 u^{(g)}\!&=&\!(\mathbbm{1}-S_{\mathrm{D}}^{-1}\widetilde{S})g\,, \label{eq:ug} \\
 f^{(g)}\!&=&\! S_{\mathrm{D}}^{-1}(\mathbbm{1}-P_{\ker S^*})\widetilde{S} g\,, \\
 Tu^{(g)}+w^{(g)}\!&=&\! P_{\ker S^*}\widetilde{S} g\,, \label{eq:Tugpwg}
\end{eqnarray}
  and moreover, as $\varepsilon\downarrow 0$,
 \begin{eqnarray}
 (\mathbbm{1}-S_{\mathrm{D}}^{-1}S^*)(u_\varepsilon-U_\varepsilon u_\varepsilon)\!&=&\! u^{(g)}+O(\varepsilon)\,, \label{eq:limtoug} \\
  \left.
  \begin{array}{r}
   2\,\ii\,\varepsilon\,u_\varepsilon^{(g)} \! \\
   2\,\ii\,\varepsilon\,U_\varepsilon u_\varepsilon^{(g)} \!
  \end{array}\right\}
  \!&=&\! T u^{(g)} + w^{(g)}+O(\varepsilon)\,, \\
   u_\varepsilon^{(g)} -U_\varepsilon u_\varepsilon^{(g)}\!&=&\! S_{\mathrm{D}}^{-1}(Tu^{(g)}+w^{(g)})+u^{(g)}+O(\varepsilon)\,, \label{eq:limuepsmUueps}\\
    \ii\,\varepsilon\, u_\varepsilon^{(g)} +  \ii\,\varepsilon\, U_\varepsilon u_\varepsilon^{(g)}\!&=&\! T u^{(g)} + w^{(g)}+O(\varepsilon)\,, \label{eq:iepsuepsplusveps} \\
  f_\varepsilon^{(g)} \!&=&\! f^{(g)} +O(\varepsilon)\,, \label{eq:fepstofingen}\\
  \overline{S} f_\varepsilon^{(g)} \!&=&\! \overline{S} f^{(g)}+O(\varepsilon) \label{eq:Sstarfepstofingen}
 \end{eqnarray}
  in the Hilbert norm of $\cH$. \eqref{eq:limuepsmUueps}-\eqref{eq:iepsuepsplusveps} and \eqref{eq:fepstofingen}-\eqref{eq:Sstarfepstofingen} can be combined, respectively, as
   \begin{eqnarray}
    u_\varepsilon^{(g)} -U_\varepsilon u_\varepsilon^{(g)}\!&=&\! S_{\mathrm{D}}^{-1}(Tu^{(g)}+w^{(g)})+u^{(g)}+O(\varepsilon)\,, \label{eq:uepsUepsgraphn} \\
    f_\varepsilon^{(g)} \!&=&\! f^{(g)} +O(\varepsilon) \label{eq:fepsfgraphn} 
   \end{eqnarray}
  in the graph norm of $\widetilde{S}$ (hence also in the graph norm of $S^*$).
\end{theorem}

 The combination of \eqref{eq:uvareps} and \eqref{eq:Uuvareps} yields an explicit characterisation of the operator label $U_\varepsilon$ in terms of $\widetilde{S}$ and $P_{\ker(S^*\pm\ii\varepsilon\mathbbm{1})}$, beside the already known unique association $\widetilde{S}\mapsto U_\varepsilon^{(\widetilde{S})}$. Let us formulate this for a generic spectral point $z\in\mathbb{C}^+$, instead of just $z=\ii\varepsilon$.

 \begin{corollary}
  Assume \eqref{eq:assumptions} and let $\widetilde{S}$ be a self-adjoint extension of $S$. For every $z\in\mathbb{C}^+$ and every $g\in\mathcal{D}(\widetilde{S})$ the extension parameter $U\equiv U^{(\widetilde{S})}:\ker(S^*-z\mathbbm{1})\to\ker(S^*-\overline{z}\mathbbm{1})$ associated to $\widetilde{S}$ in the von Neumann parametrisation  \eqref{eq:vNparam} acts as
  \begin{equation}\label{eq:Uepsreconstr}
   U_\varepsilon^{(\widetilde{S})} P_{\ker(S^*-z\mathbbm{1})}(\widetilde{S}-\overline{z}\mathbbm{1})g\;=\;P_{\ker(S^*-\overline{z}\mathbbm{1})}(\widetilde{S}-z\mathbbm{1})g\,.
  \end{equation}
 \end{corollary}

  Formula \eqref{eq:Uepsreconstr} is conceptually relevant in that, as commented already, the $U$-parameter in the von Neumann extension classification is not referred to the canonical parameter of a distinguished extension, so one can only have expressions for the action of the $U$-parameter directly in terms of the corresponding $\widetilde{S}\equiv S_U$.

  On a related footing, the limit \eqref{eq:limuepsmUueps} provides an approximation (in the norm of $\cH$) of the action of $\mathbbm{1}-U_\varepsilon$ on a generic $u_\varepsilon^{(g)}\in\ker(S^*-\ii\varepsilon\mathbbm{1})$ in terms of the action of $T$ on the corresponding $u^{(g)}$ and in terms of the corresponding $w^{(g)}$.
   In fact, the combination of \eqref{eq:ug} and \eqref{eq:Tugpwg} with the limit \eqref{eq:limuepsmUueps} yields
   \begin{equation}\label{eq:1mUepslim2}
    (\mathbbm{1}-U_\varepsilon)u_\varepsilon^{(g)}\;\stackrel{\varepsilon\downarrow 0}{=}  \;\big(\mathbbm{1}-S_{\mathrm{D}}^{-1}(\mathbbm{1}-P_{\ker S^*})\widetilde{S}\big)g+O(\varepsilon)\,.	
   \end{equation}
  Formulas \eqref{eq:limuepsmUueps}, \eqref{eq:Uepsreconstr}, or \eqref{eq:1mUepslim2} clearly provide an answer, implicit but possibly exploitable in applications, to the question concerning recovering the extension parameter $U_\varepsilon\equiv U_\varepsilon^{(\widetilde{S})}$ in terms of $\widetilde{S}$ or the extension parameter $T\equiv T^{(\widetilde{S})}$.

  Concerning the converse question, from $U_\varepsilon^{(\widetilde{S})}$ to $T^{(\widetilde{S})}$, one may say the following.
  
  \begin{theorem}\label{thm:recoverT}
    Assume \eqref{eq:assumptions} and let $\widetilde{S}$ be a self-adjoint extension of $S$. For arbitrary $\varepsilon>0$ let $U_\varepsilon\equiv U_\varepsilon^{(\widetilde{S})}$ be the extension parameter of $\widetilde{S}$ given, by \eqref{eq:vNparameps}, so that an arbitrary $g\in\mathcal{D}(\widetilde{S})$ is decomposed as 
    \begin{equation}\label{eq:gdecompeps}
  g\;=\;f_\varepsilon^{(g)}+u_\varepsilon^{(g)}- U_\varepsilon u_\varepsilon^{(g)}
 \end{equation}
 for $f_\varepsilon^{(g)}\in\mathcal{D}(\overline{S})$ and $u_\varepsilon^{(g)}\in\ker(S^*-\ii\varepsilon\mathbbm{1})$ uniquely determined by $g$. Then the domain and the matrix elements of the extension parameter $T\equiv T^{(\widetilde{S})}$ of $\widetilde{S}$ given by \eqref{eq:ST-2again} satisfy
 \begin{equation}\label{eq:TfromU}
  \begin{split}
   \mathcal{D}(T)\;&=\;\left\{ u^{(g)}\in\ker S^*\,\left|
   \begin{array}{c}
    u^{(g)}:=\displaystyle\lim_{\varepsilon\downarrow 0}(\mathbbm{1}-S_{\mathrm{D}}^{-1}\widetilde{S})(u_\varepsilon^{(g)}- U_\varepsilon u_\varepsilon^{(g)}) \\
    \textrm{ for some }g\in\mathcal{D}(\widetilde{S})
   \end{array}
   \!\right.\right\}, \\
   \langle u^{(g)},Tu^{(g)}\rangle_{\cH}\;&=\;\lim_{\varepsilon\downarrow 0} \;\ii\,\varepsilon\,\big\langle (\mathbbm{1}-S_{\mathrm{D}}^{-1}\widetilde{S})(u_\varepsilon^{(g)}- U_\varepsilon u_\varepsilon^{(g)})\,,\,(u_\varepsilon^{(g)}+ U_\varepsilon u_\varepsilon^{(g)})\big\rangle_{\cH}\,.
  \end{split}
 \end{equation}
%
%
%
%

%
%
%
%
\end{theorem}

 \begin{proof}
  The vectors that form $\mathcal{D}(T)$ are identified by the limit \eqref{eq:limtoug}, in view of \eqref{eq:gdecomposedaccordingly}. Moreover, \eqref{eq:ST-2again} and \eqref{eq:gdecomposedaccordingly} imply
  \[
   \langle u^{(g)},Tu^{(g)}\rangle_{\cH}\;=\;\langle u^{(g)},Tu^{(g)}+w^{(g)}\rangle_{\cH}\,,
  \]
   and expressing the two entries of the scalar product in the r.h.s.~above by means of the limits \eqref{eq:limtoug} and \eqref{eq:iepsuepsplusveps} yields the second line of \eqref{eq:TfromU}.
 \end{proof}

 \begin{remark}
  In conclusion, it is worth remarking the following point of view, that we kept in the course of our analysis.
  The `absolute' and `relative' self-adjoint extension parametrisations of a symmetric operator satisfying \eqref{eq:assumptions} have extension labels, the operators $U^{(\widetilde{S})}$ and $T^{(\widetilde{S})}$ respectively, which are uniquely identified by each extension $\widetilde{S}$, and as such they identify each other. Observe that this fact involves two extension parameters that not only are different in nature (a unitary vs a self-adjoint, acting in different spaces), but also have different conceptual roles in the respective classification schemes: in the `relative' one, $T^{(\widetilde{S})}$ identifies $\widetilde{S}$ in terms of the distinguished extension $S_\mathrm{D}$, whereas in the `absolute' classification $U^{(\widetilde{S})}$ parametrises $\widetilde{S}$ with no a priori reference. Thanks to exploiting the two extension parametrisations respectively at the spectral points $z=\ii\varepsilon$ and $z=0$, namely \eqref{eq:vNparameps} and \eqref{eq:ST-2again}, we could find quantitative limits as $\varepsilon\downarrow 0$ that connect for each self-adjoint extension the corresponding operator parameters $U_\varepsilon^{(\widetilde{S})}$ and $T^{(\widetilde{S})}$, namely  
  \eqref{eq:limuepsmUueps}, \eqref{eq:1mUepslim2}, and \eqref{eq:TfromU}. Of course, the choice of generic spectral points $z$ and $0$, while not yielding explicit limits as the above ones, does not alter the fact that to each $T^{(\widetilde{S})}$ one associates uniquely the corresponding $U^{(\widetilde{S})}$, and vice versa. The (implicit) correspondence is clear from \eqref{eq:vNparam} and \eqref{eq:ST-2} and reads:
  \begin{itemize}
   \item given $T$, and running $u$ in $\mathcal{D}(T)$, as well as $w$ in $\mathcal{D}(T)\cap \ker S^*$ and $f$ in $\mathcal{D}(\overline{S})$, build all possible $g=f+S_{\mathrm{D}}^{-1}(Tu+w)+u$;
   \item correspondingly, running over all such $g$'s, $P_{\ker(S^*-z\mathbbm{1})}(\widetilde{S}-\overline{z}\mathbbm{1})g$ exhausts $\ker(S^*-z\mathbbm{1})$ and the action of $U^{(\widetilde{S})}:\ker(S^*-z\mathbbm{1})\to\ker(S^*-\overline{z}\mathbbm{1})$ on any such vector of  $\ker(S^*-z\mathbbm{1})$ is finally given by \eqref{eq:Uepsreconstr}.   
  \end{itemize}
 \end{remark}

 \section{Two examples}\label{eq:examples}
 
 In the first example of this Section we want to elucidate the limits \eqref{eq:uepsUepsgraphn}-\eqref{eq:fepsfgraphn}.

 Consider the operator $S$ acting in the Hilbert space $\cH=L^2(\mathbb{R}^+)$ as
 \begin{equation}\label{eq:defSex1}
  \mathcal{D}(S)\;=\;C^\infty_c(\mathbb{R}^+)\,,\qquad S=-\frac{\ud^2}{\ud x^2}+\mathbbm{1}\,.
 \end{equation}
  It is standard to see (see, e.g., \cite[Section 6.2.2]{GTV-2012}, \cite[Example 14.15]{schmu_unbdd_sa}, \cite[Section 7.1]{GMO-KVB2017}) that $S$ is symmetric, densely defined, lower semi-bounded with $\mathfrak{m}(S)=1$, and
\begin{eqnarray}
 \mathcal{D}(\overline{S})\!&=&\! H^2_0(\mathbb{R}^+)\;=\;\left\{ f\in C^1(\overline{\mathbb{R}^+})\cap L^2(\mathbb{R}^+)\left|
 \begin{array}{c}
  f'\in AC(\overline{\mathbb{R}^+})\,, \\
  f(0)=0=f'(0)
 \end{array}\!
\right.\right\}, \label{eq:domSclex1} \\
\mathcal{D}(S_{\mathrm{F}})\!&=&\! H^2(\mathbb{R}^+)\cap H^1_0(\mathbb{R}^+)\;=\;\{g\in H^2(\mathbb{R}^+)\,|\,g(0)=0\}\,, \label{eq:domSFex1}\\
\mathcal{D}(S^*)\!&=&\! H^2(\mathbb{R}^+)\,,
\end{eqnarray}
 the operator $S^*$ and its symmetric restrictions acting again as $S^*=-\frac{\ud^2}{\ud x^2}+\mathbbm{1}$, in general now as a weak derivative, as compared to the classical derivative in \eqref{eq:defSex1}. Moreover, $S$ has unit deficiency index, with deficiency spaces
 \begin{equation}\label{eq:defsp-ex1}
  \ker S^*\;=\;\mathrm{span}\{e^{-x}\}\,,\qquad \ker (S^*\mp\ii\varepsilon)\;=\;\mathrm{span}\{e^{-x\sqrt{1\mp\ii\varepsilon}}\}\,,\quad\varepsilon\,>\,0\,,
 \end{equation}
 and the action of $S_\mathrm{F}^{-1}$ on $\ker S^*$ is given by
 \begin{equation}\label{eq:SFactionex1}
  S_\mathrm{F}^{-1}e^{-x}\;=\;\frac{1}{2}\,x\,e^{-x}\,.
 \end{equation}

 The self-adjoint extensions of $S$ in $\cH$ are labelled by unitaries $U_\varepsilon$ that, in view of the unital dimensionality of the spaces \eqref{eq:defsp-ex1}, act as multiplication by a phase. Explicitly, upon normalising the spanning vectors for $\ker (S^*\mp\ii\varepsilon)$, each such $U_\varepsilon$ is given by
 \begin{equation}\label{eq:Uepsex1}
  \begin{split}
   U_\varepsilon\,:\ker (S^*-\ii\varepsilon)\,&\to\,\ker (S^*+\ii\varepsilon)\,, \\
   U_\varepsilon\big(\sqrt{2}\,\kappa_\varepsilon\, e^{-x\sqrt{1-\ii\varepsilon}}\big)\;&=\;e^{\ii\theta_\varepsilon}\sqrt{2}\,\kappa_\varepsilon\,e^{-x\sqrt{1+\ii\varepsilon}}
  \end{split}
 \end{equation}
  for some $\theta_\varepsilon\in[0,2\pi)$, where
  \begin{equation}\label{eq:kappaeps}
   \kappa_\varepsilon\;:=\;(1+\varepsilon^2)^{\frac{1}{8}}\sqrt{\cos\,\frac{\arctan\varepsilon}{2}}\,.
  \end{equation}

  It is instructive now to identify the operator label $U_\varepsilon$ -- in practice, the real number $\theta_\varepsilon$ -- that selects the Friedrichs extension $S_{\mathrm{F}}$. Indeed, in the von Neumann extension parametrisation \eqref{eq:vNparameps} $S_{\mathrm{F}}$ and its extension label have no distinguished position.

  Decomposing an arbitrary $g\in\mathcal{D}(S_{\mathrm{F}})$ as in \eqref{eq:gdecompeps}, one has, in view of \eqref{eq:Uepsex1},
  \begin{equation}\label{eq:gexpex1}
   g\;=\;f_\varepsilon^{(g)}+c_\varepsilon^{(g)}\sqrt{2}\,\kappa_\varepsilon\,e^{-x\sqrt{1-\ii\varepsilon}}-e^{\ii\theta_\varepsilon}\,c_\varepsilon^{(g)}\sqrt{2}\,\kappa_\varepsilon\,e^{-x\sqrt{1+\ii\varepsilon}}
  \end{equation}
  for some $f_\varepsilon^{(g)}\in\mathcal{D}(\overline{S})$ and $c_\varepsilon^{(g)}\in\mathbb{C}$. On account of \eqref{eq:domSFex1} one must have $g(0)=0$; imposing this in \eqref{eq:gexpex1}, and exploiting the regularity properties \eqref{eq:domSclex1} for $f_\varepsilon^{(g)}$, one finds
  \begin{equation}\label{eq:thetaepsex1}
   g\,=\,0\qquad\Rightarrow\qquad e^{\ii\theta_\varepsilon}\,\equiv\,1 \qquad\Rightarrow\qquad \theta_\varepsilon\,\equiv\,0
  \end{equation}
   (a condition \emph{independent} on $g$). This shows that the extension label $U_\varepsilon^{(S_{\mathrm{F}})}$ relative to the Friedrichs extension of $S$ is the unitary \eqref{eq:Uepsex1} whose $\theta_\varepsilon^{(S_{\mathrm{F}})}$ is given by \eqref{eq:thetaepsex1}. As a consequence, the decomposition \eqref{eq:gexpex1} for a generic $g\in\mathcal{D}(S_{\mathrm{F}})$ reads
     \begin{equation}\label{eq:gexpex1-2}
   g\;=\;f_\varepsilon^{(g)}+\sqrt{2}\,c_\varepsilon^{(g)}\kappa_\varepsilon\big(e^{-x\sqrt{1-\ii\varepsilon}}-e^{-x\sqrt{1+\ii\varepsilon}}\big)\,,\qquad g\in\mathcal{D}(S_{\mathrm{F}})\,.
  \end{equation}
  
  It is also instructive to compare the two extension parametrisations \eqref{eq:vNparameps} and \eqref{eq:ST-2again}, in the limit $\varepsilon\downarrow 0$, for the case of the self-adjoint extension $S_{\mathrm{F}}$. Recall that in the latter scheme, the extension parameter for $S_{\mathrm{F}}$ is canonically identified, and is precisely `$T=\infty$' ($\mathcal{D}(T)=\{0\}$). Based on this, and on the explicit action
  \eqref{eq:SFactionex1} of $S_\mathrm{F}^{-1}$ on $\ker S^*$ in this case, any $g\in\mathcal{D}(S_{\mathrm{F}})$ decomposes as 
  \begin{equation}\label{gSfex1}
   g\;=\;f^{(g)}+\frac{\,c^{(g)}}{2}\,x\,e^{-x}
  \end{equation}
  for some $f^{(g)}\in\mathcal{D}(\overline{S})$ and $c^{(g)}\in\mathbb{C}$. By equating \eqref{eq:gexpex1-2} and \eqref{gSfex1} for the same arbitrary $g\in\mathcal{D}(S_{\mathrm{F}})$ one finds
  \begin{equation}
   \begin{split}
     \mathcal{D}(\overline{S})\;\ni\;h_\varepsilon^{(g)}\;&:=\;f_\varepsilon^{(g)}-f^{(g)} \\
     &\,=\;\sqrt{2}\,c_\varepsilon^{(g)}\kappa_\varepsilon\big(e^{-x\sqrt{1-\ii\varepsilon}}-e^{-x\sqrt{1+\ii\varepsilon}}\big)-\frac{1}{2}\,c^{(g)}x\,e^{-x}
   \end{split}
  \end{equation}
   and for $h_\varepsilon^{(g)}$ equation \eqref{eq:domSclex1} above prescribes both $h_\varepsilon^{(g)}(0)=0$, which has been already implemented by means of \eqref{eq:thetaepsex1}, and
   \[
    0\;=\;(h_\varepsilon^{(g)})'(0)\;=\;\sqrt{2}\,c_\varepsilon^{(g)}\kappa_\varepsilon(-\sqrt{1-\ii\varepsilon}+\sqrt{1+\ii\varepsilon})-\frac{\,c^{(g)}}{2}\,,
   \]
  whence
  \begin{equation}\label{eq:fepscepsfc-ex1}
   \begin{split}
    c_\varepsilon^{(g)}\;&=\;c^{(g)}\,\frac{\,\sqrt{1+\ii\varepsilon}+\sqrt{1-\ii\varepsilon}\,}{\,2^{\frac{5}{2}}\,\ii\,\varepsilon\,\kappa_\varepsilon}\,, \\
    f_\varepsilon^{(g)}\;&=\;g-\frac{\,c^{(g)}}{2}\,\frac{\,\sqrt{1+\ii\varepsilon}+\sqrt{1-\ii\varepsilon}\,}{2}\,\frac{\,e^{-x\sqrt{1-\ii\varepsilon}}+e^{-x\sqrt{1+\ii\varepsilon}}\,}{\ii\,\varepsilon}
   \end{split}
  \end{equation}
  (the latter equation following from the former and from \eqref{eq:gexpex1-2}). Thus, \eqref{eq:fepscepsfc-ex1} connects at any finite $\varepsilon>0$ the components of $g$ in the von Neumann representation of $\mathcal{D}(S_{\mathrm{F}})$ with the components of the same $g$ in the canonical representation induced by the boundary triplet \eqref{eq:thetriplet}.

  Observe that $c_\varepsilon^{(g)}\sim\varepsilon^{-1}$, which confirms that the representative
  \[
   u_\varepsilon^{(g)}\;=\;c_\varepsilon^{(g)}\sqrt{2}\,\kappa_\varepsilon\,e^{-x\sqrt{1-\ii\varepsilon}}
  \]
  of $g$ in the decomposition \eqref{eq:gdecompeps} (see \eqref{eq:gexpex1} above) does indeed blow up in the $L^2$-norm as $\varepsilon\downarrow 0$, as commented in general in the end of Section \ref{sec:impretriplet}, and so too does 
  \[
   U_\varepsilon^{(S_{\mathrm{F}})}u_\varepsilon^{(g)}\;=\;c_\varepsilon^{(g)}\sqrt{2}\,\kappa_\varepsilon\,e^{-x\sqrt{1+\ii\varepsilon}}\,.
  \]
  Such a divergence cancels out in the difference
  \[
   u_\varepsilon^{(g)}-U_\varepsilon^{(S_{\mathrm{F}})}u_\varepsilon^{(g)}\;=\;\frac{\,c^{(g)}}{2}\,\frac{\,\sqrt{1+\ii\varepsilon}+\sqrt{1-\ii\varepsilon}\,}{2}\,\frac{\,e^{-x\sqrt{1-\ii\varepsilon}}-e^{-x\sqrt{1+\ii\varepsilon}}\,}{\ii\,\varepsilon}
  \]
  (obtained by comparing the second equation in \eqref{eq:fepscepsfc-ex1} with \eqref{eq:gdecompeps}): indeed, from the $x$-point-wise limit $(\ii\varepsilon)^{-1}(e^{-x\sqrt{1-\ii\varepsilon}}-e^{-x\sqrt{1+\ii\varepsilon}})\xrightarrow{\varepsilon\downarrow 0}x e^{-x}$ (where the divergence cancellation actually occurs) and from dominated convergence, one has 
  \begin{equation}
   u_\varepsilon^{(g)}-U_\varepsilon^{(S_{\mathrm{F}})}u_\varepsilon^{(g)}\;\xrightarrow{\;\varepsilon\downarrow 0\;}\;\frac{1}{2}\,c^{(g)}x\,e^{-x}\qquad \textrm{in } L^2(\mathbb{R}^+)\,,
  \end{equation}
  consistently with the general limit \eqref{eq:limuepsmUueps}, whence also, from \eqref{eq:fepscepsfc-ex1},
    \begin{equation}\label{eq:epsconv-ex1}
   \begin{split}
    f_\varepsilon^{(g)}\;&\xrightarrow{\;\varepsilon\downarrow 0\;}\;g-\frac{1}{2}\,c^{(g)}x\,e^{-x}\;=\;f^{(g)}\qquad \textrm{in } L^2(\mathbb{R}^+)\,,
   \end{split}
  \end{equation}
 consistently with \eqref{eq:fepstofingen}. The occurrence of the last two limits also in the stronger graph norm of $S_{\mathrm{F}}$, which in this case is the $H^2$-norm, is checked in a completely analogous manner, consistently with the general result \eqref{eq:uepsUepsgraphn}-\eqref{eq:fepsfgraphn}.

 In the second example we want to elucidate the scope of Theorem \ref{thm:recoverT}, and in particular the limits in \eqref{eq:TfromU}. Clearly, a meaningful example of this sort requires a symmetric operator with deficiency indices $\geqslant 2$. 

 To this aim, with respect to the Hilbert space $\cH$ given by the orthogonal direct sum
 \begin{equation}\label{eq:Hleftright}
  \cH\;:=\; L^2(\mathbb{R}^-)\oplus L^2(\mathbb{R}^+)\,,
 \end{equation}
 write $g\in\cH$ as 
 \begin{equation}
  g\;=\; \begin{pmatrix} g_- \\ g_+ \end{pmatrix}\,,\qquad g_\pm(x)\,:=\,g(x)\quad \textrm{for } x\in\mathbb{R}^\pm\,,
 \end{equation}
 and consider the operator
 \begin{equation}
  \begin{split}
  \mathcal{D}(S)\;&:=\;C^\infty_c(\mathbb{R}^-)\boxplus C^\infty_c(\mathbb{R}^+)\,, \\
  S\;&:=\;\Big(-\frac{\ud^2}{\ud x^2}+\mathbbm{1}\Big)\oplus\Big(-\frac{\ud^2}{\ud x^2}+\mathbbm{1}\Big)\,,
  \end{split}
 \end{equation}
 with the customary notation `$\boxplus$' replacing `$\oplus$' for the orthogonal sum of non-closed subspaces, with respect to the decomposition \eqref{eq:Hleftright} (see, e.g., \cite[Chapter 1]{GM-SelfAdj_book-2022}). $S$ is symmetric, densely defined, lower semi-bounded with $\mathfrak{m}(S)=1$, and reduced with respect to the decomposition \eqref{eq:Hleftright}, and moreover
 \begin{eqnarray}
  \mathcal{D}(S_{\mathrm{F}})  \!&=&\! \big( H^2(\mathbb{R}^-)\cap H^1_0(\mathbb{R}^-)\big) \boxplus \big( H^2(\mathbb{R}^+)\cap H^1_0(\mathbb{R}^+)\big)\,, \\
  \mathcal{D}(S^*)  \!&=&\! H^2(\mathbb{R}^-) \boxplus H^2(\mathbb{R}^+)\,.
 \end{eqnarray}
  $S$ has deficiency indices equal to $2$, with deficiency spaces
  \begin{eqnarray}
   \ker S^* \!&=&\! \bigg\{ \begin{pmatrix} c_-\, e^x \\ c_+\, e^{-x} \end{pmatrix}\bigg| \,c_\pm\in\mathbb{C}\bigg\} \,, \\
   \ker (S^*\mp\ii\varepsilon\mathbbm{1})  \!&=&\! \bigg\{ \begin{pmatrix} a_-\, e^{x\sqrt{1\mp\ii\varepsilon}} \\  a_+ \,e^{-x\sqrt{1\mp\ii\varepsilon}} \end{pmatrix}\bigg|\,a_\pm\in\mathbb{C}\bigg\}
  \end{eqnarray}
 for arbitrary $\varepsilon>0$.
  
  Next, consider the family of operators  $S_\alpha:= S^*\upharpoonright\mathcal{D}(S_\alpha)$, $\alpha\in\mathbb{R}$, with 
  \begin{equation}\label{eq:famSalpha}
   \mathcal{D}(S_\alpha)\;:=\;\left \{ g\in H^2(\mathbb{R}^-) \boxplus H^2(\mathbb{R}^+)\left| 
   \begin{array}{c}
    g_+(0)=g_-(0)=:g_0\,, \\
    g_+'(0)-g_-'(0)= \alpha g_0
   \end{array}\!
   \right.\right\}.
  \end{equation}
   One can check in various ways (definition of self-adjointness, criteria of self-ad\-joint\-ness, self-adjoint extension schemes) that each $S_\alpha$ is an actual self-adjoint extension of $S$. Let $U_\varepsilon^{(\alpha)}$ be the extension parameter associated with $S_\alpha$ in the von Neumann scheme \eqref{eq:vNparameps} and relative to the deficiency spaces $\ker (S^*\mp\ii\varepsilon\mathbbm{1})$. Its action, given by \eqref{eq:uvareps}-\eqref{eq:Uuvareps} above, will be computed in a moment. The goal is to apply Theorem \ref{thm:recoverT} to recover, from the knowledge of $U_\varepsilon^{(\alpha)}$, domain and action of the operator label $T^{(\alpha)}$ associated with $S_\alpha$ in the Kre{\u\i}n-Vi\v{s}ik-Birman extension parametrisation \eqref{eq:ST-2again}. One has then to set up the limits in \eqref{eq:TfromU}.

   The orthogonal projections onto the two deficiency subspaces at the spectral points $\pm\ii\varepsilon$ are given by
   \begin{equation}\label{eq:Projectionsieps}
   \begin{split}
    P_{\ker(S^*\mp\ii\varepsilon\mathbbm{1})}\;& =\; \big( 2\,\kappa_\varepsilon^2\, \big|e^{x\sqrt{1\mp\ii\varepsilon}}\big\rangle\langle e^{x\sqrt{1\mp\ii\varepsilon}}\big|\oplus\mathbb{O}\big) \\
    &\qquad +\big( \mathbb{O}\oplus 2\,\kappa_\varepsilon^2\, \big|e^{-x\sqrt{1\mp\ii\varepsilon}}\big\rangle\langle e^{-x\sqrt{1\mp\ii\varepsilon}}\big|\big)\,,
    \end{split}
   \end{equation}
   with $\kappa_\varepsilon$ given by \eqref{eq:kappaeps} and with the operator orthogonal sums in \eqref{eq:Projectionsieps} above, as well as throughout the following, referring to the Hilbert space orthogonal decomposition \eqref{eq:Hleftright}. Thus, for arbitrary $g\in\mathcal{D}(S_\alpha)$,
   \begin{equation}\label{eq:PSieps}
   \begin{split}
         P_{\ker(S^*\mp\ii\varepsilon\mathbbm{1})}&(S_\alpha\pm\ii\varepsilon\mathbbm{1})\,g  \\
         &=\;2\kappa_\varepsilon^2\begin{pmatrix}
             e^{x\sqrt{1\mp\ii\varepsilon}}\big\langle e^{x\sqrt{1\mp\ii\varepsilon}}\,,-g_-''+(1\pm\ii\varepsilon)g_-\big\rangle_- \\
             e^{-x\sqrt{1\mp\ii\varepsilon}}\big\langle e^{-x\sqrt{1\mp\ii\varepsilon}}\,,-g_+''+(1\pm\ii\varepsilon)g_+\big\rangle_+
            \end{pmatrix}\,,
   \end{split}
   \end{equation}
   with the notation $\langle\cdot,\cdot\rangle_{\pm}$ for the scalar product of $L^2(\mathbb{R}^\pm)$. Through \eqref{eq:uvareps}-\eqref{eq:Uuvareps} this provides the action of $U_\varepsilon^{(\alpha)}$, and also allows to compute
  \[
   \begin{split}
    \big(&u_\varepsilon^{(g)}-U_\varepsilon^{(\alpha)}u_\varepsilon^{(g)}\big)_+ \\
    &=\;(\ii\varepsilon)^{-1}\kappa_\varepsilon^2\Big( e^{-x\sqrt{1-\ii\varepsilon}}\big\langle  e^{-x\sqrt{1-\ii\varepsilon}}\,,-g_+''+(1+\ii\varepsilon)g_+\big\rangle_+  \\
    &\qquad\qquad\qquad - e^{-x\sqrt{1+\ii\varepsilon}}\big\langle  e^{-x\sqrt{1+\ii\varepsilon}}\,,-g_+''+(1-\ii\varepsilon)g_+\big\rangle_+\Big) \\
    &=\;\kappa_\varepsilon^2\,e^{-x\sqrt{1-\ii\varepsilon}}\big \langle e^{-x\sqrt{1-\ii\varepsilon}}\,, 2 g_+\big\rangle_+ \\
    &\qquad +\kappa_\varepsilon^2\,e^{-x\sqrt{1-\ii\varepsilon}}\Big\langle\frac{\,e^{-x\sqrt{1-\ii\varepsilon}}-e^{-x\sqrt{1+\ii\varepsilon}}\,}{-\ii\varepsilon}\,,-g_+''+(1-\ii\varepsilon)g_+\Big\rangle_{\!+} \\
    &\qquad + \kappa_\varepsilon^2\,\frac{\,e^{-x\sqrt{1-\ii\varepsilon}}-e^{-x\sqrt{1+\ii\varepsilon}}\,}{\ii\varepsilon}\,\big\langle  e^{-x\sqrt{1+\ii\varepsilon}}\,,-g_+''+(1-\ii\varepsilon)g_+\big\rangle_+ \\
    &\xrightarrow{\;\varepsilon\downarrow 0\;}\; 2\,e^{-x}\langle e^{-x},g_+\rangle_+ - e^{-x}\langle x e^{-x}\,,-g_+''+g_+\rangle_+ + x e^{-x}\langle e^{-x}\,,-g_+''+g_+\rangle_+\,,
   \end{split}
  \] 
  having used again the $(\ii\varepsilon)^{-1}(e^{-x\sqrt{1-\ii\varepsilon}}-e^{-x\sqrt{1+\ii\varepsilon}})\xrightarrow{\varepsilon\downarrow 0}x e^{-x}$ point-wise in $x$. By dominated convergence, the above limit holds in the norm of $L^2(\mathbb{R}^+)$. A completely analogous computation holds for $(u_\varepsilon^{(g)}-U_\varepsilon^{(\alpha)}u_\varepsilon^{(g)})_-$\,, so that
  \begin{equation}\label{eq:ex2prel1}
  \begin{split}
      u_\varepsilon^{(g)}-U_\varepsilon^{(\alpha)}u_\varepsilon^{(g)}\;&\xrightarrow{\;\varepsilon\downarrow 0\;}\;
      \begin{pmatrix}
       -x e^x\,\langle e^x\,,-g_-''+g_-\rangle_- \\
        x e^{-x}\,\langle e^{-x}\,,-g_+''+g_+\rangle_+
      \end{pmatrix} \\
      &\qquad\quad +
      \begin{pmatrix}
       e^x\big(\langle 2\,e^x\,,g_-\rangle_- - \langle -x e^x\,,-g_-''+g_-\rangle_-\big) \\
       e^{-x}\big(\langle 2\,e^{-x}\,,g_+\rangle_+ - \langle x e^{-x}\,,-g_+''+g_+\rangle_+\big)
      \end{pmatrix}  
  \end{split}
  \end{equation}
  in the norm of $\cH$.
  Along the same line, from \eqref{eq:uvareps}-\eqref{eq:Uuvareps}  and \eqref{eq:PSieps} one finds
  \begin{equation}\label{eq:ex2Sstarplus}
  \begin{split}
    S^*\big(u_\varepsilon^{(g)}-U_\varepsilon^{(\alpha)}u_\varepsilon^{(g)}\big)\;& =\;\ii\,\varepsilon \big(u_\varepsilon^{(g)}+U_\varepsilon^{(\alpha)}u_\varepsilon^{(g)}\big) \\
    &\xrightarrow{\;\varepsilon\downarrow 0\;}\; 
          \begin{pmatrix}
        e^x\,\langle 2\,e^x\,,-g_-''+g_-\rangle_- \\
        e^{-x}\,\langle 2\,e^{-x}\,,-g_+''+g_+\rangle_+
      \end{pmatrix}\,,
  \end{split}
  \end{equation}
  and by continuity of $S_{\mathrm{F}}^{-1}$ (and \eqref{eq:SFactionex1})
  \begin{equation}\label{eq:ex2prel2}
   S_{\mathrm{F}}^{-1}S^*\big(u_\varepsilon^{(g)}-U_\varepsilon^{(\alpha)}u_\varepsilon^{(g)}\big)\;\xrightarrow{\;\varepsilon\downarrow 0\;} \begin{pmatrix}
       -x e^x\,\langle e^x\,,-g_-''+g_-\rangle_- \\
        x e^{-x}\,\langle e^{-x}\,,-g_+''+g_+\rangle_+
      \end{pmatrix}
  \end{equation}
  in the norm of $\cH$.
  Thus, combining \eqref{eq:ex2prel1} and \eqref{eq:ex2prel2},
  \begin{equation}\label{eq:ex2lhslim}
   \begin{split}
       (\mathbbm{1}-S_{\mathrm{F}}^{-1}S^*)&\big(u_\varepsilon^{(g)}-U_\varepsilon^{(\alpha)}u_\varepsilon^{(g)}\big) \\
       &\xrightarrow{\;\varepsilon\downarrow 0\;}\; \begin{pmatrix}
       e^x\big(\langle 2\,e^x\,,g_-\rangle_- - \langle -x e^x\,,-g_-''+g_-\rangle_-\big) \\
       e^{-x}\big(\langle 2\,e^{-x}\,,g_+\rangle_+ - \langle x e^{-x}\,,-g_+''+g_+\rangle_+\big)
      \end{pmatrix}
   \end{split}
  \end{equation}
  in the norm of $\cH$.

  Before plugging the limits \eqref{eq:ex2Sstarplus} and \eqref{eq:ex2lhslim} into \eqref{eq:TfromU} it is convenient to suitably re-write the coefficients therein. Concerning \eqref{eq:ex2lhslim}, simple integration by parts yields
  \[
   \begin{split}
   \langle 2\,e^x\,,g_-\rangle_- - \langle -x e^x\,,-g_-''+g_-\rangle_-\;&=\;g_-(0) \,,\\
   \langle 2\,e^{-x}\,,g_+\rangle_+ - \langle x e^{-x}\,,-g_+''+g_+\rangle_+\;&=\;g_+(0)\,,   
   \end{split}
  \]
  whence 
  \begin{equation}\label{eq:ex2finallhs}
    (\mathbbm{1}-S_{\mathrm{F}}^{-1}S^*)\big(u_\varepsilon^{(g)}-U_\varepsilon^{(\alpha)}u_\varepsilon^{(g)}\big) \;\xrightarrow{\;\varepsilon\downarrow 0\;}\; g_0 \begin{pmatrix}
       e^x \\
       e^{-x}
      \end{pmatrix}.
  \end{equation}
  Therefore, owing to the first formula in \eqref{eq:TfromU},
  \begin{equation}\label{eq:ex1foundDT}
   \mathcal{D}(T)\;=\;\left\{\left. g_0 \begin{pmatrix}
       e^x \\
       e^{-x}
      \end{pmatrix} \right| g_0\in\mathbb{C} \right\}\,,
  \end{equation}
  that is, the Hilbert subspace of $\ker S^*$ where the extension label $T^{(\alpha)}$ of $S_\alpha$ acts in is the \emph{one-dimensional} span of $e^x\oplus e^{-x}$. On such a space the self-adjoint $T^{(\alpha)}$ necessarily acts as multiplication by a real number. The precise value of such a number will be determined in a moment by means of the second formula in \eqref{eq:TfromU}. Observe also that 
  \begin{equation}\label{eq:ex1foundDTperp}
   \ker S^*\cap \mathcal{D}(T)^\perp\;=\;\left\{\left. d \begin{pmatrix}
       -e^x \\
       e^{-x}
      \end{pmatrix} \right| d\in\mathbb{C} \right\}.
  \end{equation}
  Concerning \eqref{eq:ex2Sstarplus}, instead, integration by parts and \eqref{eq:famSalpha} yield
  \[
   \begin{split}
    \langle 2\,e^x\,,-g_-''+g_-\rangle_- \;&=\; 2 ( g_-(0)-g'_-(0)) \\
    &=\; (2+\alpha) g_0 -(2g_-'(0)+\alpha g_0)\\
    &=\; (2+\alpha) g_0 -( 2g_+'(0)-\alpha g_0)\,, \\
    \langle 2\,e^{-x}\,,-g_+''+g_+\rangle_+ \;&=\; 2 (g_+(0)+g'_+(0)) \\
    &=\; (2+\alpha) g_0 +( 2g_+'(0)-\alpha g_0)\,,
   \end{split}
  \]
  whence
  \begin{equation}\label{eq:ex2finalrhs}
   \ii\,\varepsilon \big(u_\varepsilon^{(g)}+U_\varepsilon^{(\alpha)}u_\varepsilon^{(g)}\big) \;\xrightarrow{\;\varepsilon\downarrow 0\;}\; (2+\alpha) g_0\begin{pmatrix}
       e^x \\
       e^{-x}
      \end{pmatrix}  +( 2g_+'(0)-\alpha g_0) \begin{pmatrix}
       -e^x \\
       e^{-x}
      \end{pmatrix} .
  \end{equation}

  Pick now $u^{(g)}\in \mathcal{D}(T)$. Owing to \eqref{eq:ex1foundDT}, \eqref{eq:ex1foundDTperp}, and \eqref{eq:ST-2again}, $u^{(g)}$ corresponds to functions $g\in\mathcal{D}(S_\alpha)$ of the form 
  \[
   g\;=\;\begin{pmatrix} f^{(g)}_- \\ f^{(g)}_+ \end{pmatrix}+S_{\mathrm{F}}^{-1}\left[ T^{(\alpha)} c^{(g)} \begin{pmatrix}
       e^x \\
       e^{-x}
      \end{pmatrix}+ d^{(g)} \begin{pmatrix}
       -e^x \\
       e^{-x}
      \end{pmatrix}\right] +  c^{(g)} \begin{pmatrix}
       e^x \\
       e^{-x}
      \end{pmatrix}
  \]
  for $f^{(g)}_-\oplus f^{(g)}_+\in\mathcal{D}(\overline{S})$ and $c^{(g)},d^{(g)}\in\mathbb{C}$, whence $u^{(g)}=c^{(g)}\binom{e^x}{e^{-x}}$, $c^{(g)}=g_0$, and 
  \begin{equation}\label{eq:uu}
      \langle u^{(g)},u^{(g)}\rangle_{\cH}\;=\;|g_0|^2\,.
  \end{equation}
   On the other hand, the second formula in \eqref{eq:TfromU}, on account of \eqref{eq:ex2finallhs} and \eqref{eq:ex2finalrhs}, yields
   \begin{equation}\label{eq:uTu}
    \begin{split}
      \langle u^{(g)},&T^{(\alpha)}u^{(g)}\rangle_{\cH} \\
      &=\;\left \langle g_0 \begin{pmatrix}
       e^x \\
       e^{-x}
      \end{pmatrix} , (2+\alpha) g_0\begin{pmatrix}
       e^x \\
       e^{-x}
      \end{pmatrix}  +( 2g_+'(0)-\alpha g_0) \begin{pmatrix}
       -e^x \\
       e^{-x}
      \end{pmatrix}\right\rangle_{\!\cH} \\
      &=\;(2+\alpha) |g_0|^2\,.
    \end{split}
   \end{equation}
  Comparing \eqref{eq:uu} and \eqref{eq:uTu} finally shows that $T^{(\alpha)}$ is the operator of multiplication by $(2+\alpha)$ acting on the domain $\mathcal{D}(T)=\mathrm{span}\{ e^x\oplus e^{-x}\}$.


\def\cprime{$'$}

\end{document}